\documentclass[12pt]{article}
\usepackage[T2A]{fontenc}
\usepackage[cp866]{inputenc}
\usepackage{amssymb,amsmath}
\usepackage{amsthm}
\usepackage{graphicx}
\usepackage{cite}
\usepackage{url}
\usepackage{srcltx}

\newtheorem{proposition}{Proposition}
\newtheorem{lemma}{Lemma}
\newtheorem{theorem}{Theorem}
\newtheorem{corollary}{Corollary}

\theoremstyle{definition}

\newtheorem{example}{Example}

\newcommand{\cov}{\mathop{\rm cov}}
\def\MBP{\operatorname{MBP}\nolimits}
\def\MBPRE{\operatorname{MBPRE}\nolimits}
\def\MBPPLRE{\operatorname{MBPPLRE}\nolimits}

\begin{document}

\title{Maximal Branching Processes in \\ Random Environment\footnote{This is translation (from Russian) of https://doi.org/10.14357/19922264180206}}

\author{A. V. Lebedev\thanks{Department of Mathematical Statistics and Random
Processes, Faculty of Mechanics and Mathematics, Lomonosov Moscow State University,
Moscow, Russia; e-mail: \texttt{avlebed@yandex.ru}}}

\maketitle

\begin{abstract}
The work continues the author's many-year research in theory of maximal branching
processes, which are obtained from classical branching processes by replacing the
summation of descendant numbers with taking the maximum. One can say that in each
generation, descendants of only one particle survive, namely those of the particle
that has the largest number of descendants. Earlier, the author generalized processes
with integer values to processes with arbitrary nonnegative values, investigated
their properties, and proved limit theorems. Then processes with several types of
particles were introduced and studied. In the present paper we introduce the notion
of maximal branching processes in random environment (with a single type of
particles) and an important case of a ``power-law'' random environment. In the
latter case, properties of maximal branching processes are studied and the ergodic
theorem is proved. As applications, we consider gated infinite-server queues.

\end{abstract}

\section{Introduction}

Classical objects of research in theory of stochastic processes are Galton--Watson
branching processes (with a single type of particles and discrete time) \cite{Har}.
Their extremal counterparts are referred to as maximal branching processes (MBPs).
Namely, summation of the number of particle descendants (when finding the size of the
next generation) is replaced with taking the maximum.

Let us recall the history of the question. MBPs were introduced and studied by
J.~Lamperti \cite{Lamp1,Lamp2} in 1970--1972, but were later completely abandoned by
researchers (though mentioned in the survey \cite{VatZub2}). A new stage of studying
the MBPs was started by A.V.~Lebedev in 2001. Processes with integer values were
generalized to processes with arbitrary nonnegative values \cite{Leb-2005d}. First,
MBPs with particles of a single type were studied (see the survey \cite{Leb-2009a}),
and then those with several types of particles (multi-type MBPs)~\cite{Leb-2012a}. At
present, the most complete overview of results and literature is given in
\cite[Chs.~4 and~5]{LebDiss}. Until recently, the author's studies on this subject
remained solitary. Only in 2012 they were unexpectedly pursued in the work of foreign
researchers O.~Aydogmus, A.P.~Ghosh, S.~Ghosh, and A.~Roitershtein \cite{CMBP}, who
introduced colored maximal branching processes. Their difference from multi-type MBPs
is that types (colors) of particles are determined after forming a generation, at
random, and the type influences further fecundity. Another distinction from the
author's approach was considering solely processes that go to infinity.

Let us recall basic notions and properties of MBPs with a single particle type.

Consider random processes with values in $\mathbb{Z}_+$ defined stochastically by
recurrence relations of the form
\begin{equation}\label{f1-41}
Z_{n+1}=\bigvee_{m=1}^{Z_n}\xi_{m,n},
\end{equation}
where $\bigvee$ stands for the maximum operation and where $\xi_{m,n}$, $m\ge 1$,
$n\ge 0$, are independent random variables with a common distribution $F$ on
$\mathbb{Z}_+$. We assume (as in the case of summation) that the result of taking the
maximum ``zero times'' (when $Z_n=0$) is zero.

One can say (by the analogy with Galton--Watson processes) that in each generation of
a maximal branching process descendants of only one particle survive, namely of the
particle that has the largest number of descendants. It is also clear that the set of
possible values of an MBP (for $n\ge 1$) coincides with the set of possible values of
the number of descendants. It follows from \eqref{f1-41} that the process is a
homogeneous Markov chain on this set.

Another interpretation of an MBP can be proposed in queueing theory, by considering
gated multi-server queues. These are queues with infinitely many servers where access
of customers to service is regulated by a gate. The gate is assumed to be open only
when all severs are free. Customers enter a queue with infinitely many waiting
places, and servicing is performed in stages. At the beginning of a stage, when the
gate opens, all customers in the queue instantly get access to servers and then are
being served in parallel and independently until all servers become free. At the
moment when all servers become free, the gate opens again for a new batch of
customers (that have arrived during this time period) and the next stage.

Note that this queueing system is very easy to manage: there is no need to keep a
permanent record of arriving and leaving customers, free and busy servers, etc. The
allocation of customers to servers (which are all free at that moment) is made once
and simultaneously at the beginning of each stage. Another advantage of a gated
system may manifest itself in a situation where customers in the queue and servers
are somehow separated from each other and establishment of communication requires
certain costs. For example, it may be disadvantageous (or impossible) to keep the
communication channel on all the time, and short connections from time to time may be
preferable.

Of course, any infinite-linear system is only an approximation to the case where the
real number of servers is large. On the other hand, it makes sense to study such
systems to estimate different characteristics of the service quality (which for any
queue with finitely many servers can only be worse).

Consider such a queue with discrete time, and assume that there is exactly one
arrival at each time instant. Then, since the servers operate in parallel, the time
required for servicing a current batch of arrivals (and hence, the number of arrivals
in the next batch) equals the maximum of their service times. Thus, if we denote by
$Z_n$ the duration of the $n$th stage and by $\xi_{m,n}$ the service times of
arrivals within it, we obtain exactly \eqref{f1-41}.

Discrete-time gated infinite-server queues with Poisson arrival flows were analyzed
in [10--13] (using other methods) and by the author in \cite{Leb-2003,Leb-2004}. In
this case it is necessary to specify what happens if a gate opens while the queue is
empty. It is most natural to assume that the queue waits for a new customer to
arrive, from which the next stage begins.

Note that models with parallel data processing have become very popular in recent
years due to the development of cloud computing technologies. At the same time, there
arises a need to study maxima of random variables. As a recent work on this topic,
note \cite{Obl}.

MBPs were introduced in \cite{Lamp1} (in connection with long-range percolation
models), where recurrence criteria for them were also obtained. Namely (assuming
$F(0)=0$): provided that
\begin{equation}\label{ulam}
\limsup_{x\to+\infty}x(1-F(x))<e^{-\gamma},
\end{equation}
where $\gamma=0,577\ldots\strut$ is the Euler constant, the chain $\{Z_n\}$ is
positively recurrent; and vice versa, if
$$
\liminf_{x\to+\infty}x(1-F(x))>e^{-\gamma},
$$
then $Z_n\to +\infty$, $n\to\infty$, almost surely (a.s.).

Then, in \cite{Lamp2}, the critical case $x(1-F(x))\to e^{-\gamma}$, $x\to+\infty$,
was considered taking into account further terms of the tail expansion at infinity.
It was shown that if $(e^\gamma x(1-F(x))-1)\ln x\to d$, $x\to+\infty$, then the
process is recurrent when $d<\pi^2/12$ and goes to infinity a.s.\ when $d>\pi^2/12$.

According to \eqref{f1-41} and the assumption on the case of $Z_n=0$, the process has
transition probabilities
$$
{\bf P}(Z_{n+1}\le j\,|\,Z_n=i)=F^i(j),\quad i,j\in \mathbb{Z}_+
$$
(where we assume $0^0=1$), which suggested in \cite{Leb-2005d} to consider Markov
chains on an arbitrary measurable set $T\subset \mathbb{R}_+$ with transition
probabilities
\begin{equation}\label{obo}
{\bf P}(Z_{n+1}\le y\,|\,Z_n=x)=F^x(y),\quad x,y\in T,
\end{equation}
where $F$ is also supported on $T$.

Such processes can be considered both in their own right and as limit processes (in
any sense) for MBPs on $\mathbb{Z}_+$ (normalized in a certain way). For instance,
they can be used to describe the behavior of gated infinite-server queues, in
particular, limit behavior under heavy traffic conditions, etc.

In particular, for a continuous-time gated infinite-server queue with Poisson
arrivals, the sequence of stage durations over a busy period satisfies \eqref{obo}
with $F(x)=\exp\{-\lambda {\bar B}(x)\}$, $x\ge 0$, where $\lambda$ is the arrival
flow density, $B(x)$ is the distribution function for the service time of a single
customer, and ${\bar B}(x)=1-B(x)$. Indeed, denote the duration of the $n$th stage by
$Z_n$. Given that $Z_n=x$, the number of arrivals in this stage is Poissonian with
parameter $\lambda x$, and $Z_{n+1}$ is the maximum of this (random) number of
independent random variables with distribution $B$. Thus,
$$
{\bf P}(Z_{n+1}\le y\,|\,Z_n=x)= \sum_{k=0}^\infty\frac{(\lambda x)^k}{k!}e^{-\lambda
x}B(y)^k= \exp\{-\lambda x {\bar B}(y)\}=F(y)^x.
$$

To describe the system over the whole time horizon, one needs MBPs with immigration
at the vanishing moment, etc.; see \cite[Section~4.3]{LebDiss}.

Below we will speak about maximal branching processes on $T$ and denote them by
$\MBP(T)$. Note that a similar generalization for Galton--Watson processes are
Ji\v{r}ina processes \cite{Jr2} (with a continuous state set and discrete time);
however, in this case $T$ can be any measurable subset of $\mathbb{R}_+$.

Equation \eqref{f1-41} for such MBPs does not hold in the general case, but according
to \eqref{obo} they admit an equivalent representation by a stochastic recurrence
sequence of the form
\begin{equation}\label{obo2}
Z_{n+1}=
\begin{cases}
\displaystyle F^{-1}(U_{n+1}^{1/Z_n}), & Z_n>0,\\ \displaystyle 0, & Z_n=0,
\end{cases}\qquad n\ge 0,
\end{equation}
where $F^{-1}(y)=\inf\{x:\: F(x)\ge y\}$; $U_n$, $n\ge 1$, are i.i.d.\ random
variables on $(0,1)$; and $Z_0\ge 0$ is independent of them. In this case the
distribution $F$ is still called the distribution of the number of (direct)
descendants.

In \cite{Leb-2005d}, a number of MBP properties (similarity condition, association,
monotonicity in parameters, degeneration condition) were obtained and an ergodic
theorem was proved.

Note that zero is always an absorbing state for an MBP. Thus, for
$\MBP(\mathbb{Z}_+)$ under condition \eqref{ulam} and $F(0)>0$, this leads to
degeneration a.s.\ \cite{Lamp1}. For an $\MBP(T)$, if $F(0)=0$, zero can simply be
excluded from the state set by considering a process with nonzero initial condition.
However, if zero is a limit point of $T$, the possibility of asymptotic convergence
to it as $n\to\infty$ remains. The following theorem provides sufficient conditions
to eliminate the possibilities for the process to go to either zero or infinity and
to make it ergodic. Here and in what follows we assume Harris ergodicity
\cite[ch.~1]{Borov}.

\smallskip
\textbf{Theorem A} \cite[Theorem 1]{Leb-2005d}. \emph{If for an\ $\MBP(T)$\textup,
$T\subset (0,+\infty)$\textup, the conditions\/ \eqref{ulam} and
$$
\liminf_{x\to 0}x(-\ln F(x))>e^{-\gamma}
$$
are fulfilled\textup, then the process is ergodic.}

\smallskip
Now we define MBPs in random environment (MBPREs).

In applications, ``random environment'' may describe various factors of natural,
technical, or social nature randomly varying in time and affecting the system (for
instance, speeding up or slowing down the operation of a queue).

Let us be given a sequence $F_l$, $l\ge 1$, of distributions on $\mathbb{Z}_+$; a
collection of independent random variables $\xi_{m,n,l}$, $m\ge 1$, $n\ge 0$, $l\ge
1$, with distributions $F_l$, $l\ge 1$; and a sequence of independent random
variables $\nu_n$, $n\ge 0$, with common distribution $G$ on $\mathbb{N}$ and
independent of the $\xi_{m,n,l}$, $m\ge 1$, $n\ge 0$, $l\ge 1$. Then an
$\MBPRE(\mathbb{Z}_+)$ can be defined with the help of the stochastic recurrence
relation of the form
\begin{equation}\label{f1-41-mvpss}
Z_{n+1}=\bigvee_{m=1}^{Z_n}\xi_{m,n,\nu_n}.
\end{equation}
The random environment is reflected here by the random variables $\nu_n$, $n\ge 1$,
on which the choice of a distribution $F_l$, $l\ge 1$, of the number of particle
descendants in each step depends.

According to \eqref{f1-41-mvpss} and the assumption on the case of $Z_n=0$, the
process has transition probabilities
$$
{\bf P}(Z_{n+1}\le j\,|\,Z_n=i)={\bf E}(F^i_\nu(j)),\quad i,j\in \mathbb{Z}_+,
$$
where $\nu$ has distribution $G$, which suggests to consider a process on an
arbitrary measurable set $T\subset \mathbb{R}_+$ with transition probabilities
\begin{equation}\label{obo-mvpss}
{\bf P}(Z_{n+1}\le y\,|\,Z_n=x)={\bf  E}(F^x_\nu(y)),\quad x,y\in T,
\end{equation}
where a family of distributions $F_s$, $s>0$, on $T$ and a random variable $\nu$ with
distribution~$G$ on $(0,+\infty)$ are assumed. In this way, we define $\MBPRE(T)$.

Such processes can be considered both in their own right and as limit processes (in
any sense) for the $\MBPRE(\mathbb{Z}_+)$ processes (normalized in a certain way).

Below we will study a class of processes with a power-law property $F_s(x)=F^s(x)$,
$s>0$, where $F$ is a distribution on $T$. In this case we will speak about a
``power-law'' random environment and denote the corresponding processes by
$\MBPPLRE(T)$.

Introduce the Laplace--Stieltjes transform $\varphi(u)={\bf E}e^{-u\nu}$. Then
equation \eqref{obo-mvpss} takes the form
\begin{equation}\label{obo-mvpss-2}
{\bf P}(Z_{n+1}\le y\,|\,Z_n=x)=\varphi(-x\ln F(y)),\quad x,y\in T.
\end{equation}
We face up the problem of analyzing properties of such processes.

\section{Basic Properties}

First of all, note that for an MBPPLRE we also have a constructive representation,
which follows from \eqref{obo-mvpss-2}, namely
\begin{equation}\label{obo2-mvpss}
Z_{n+1}=
\begin{cases}
\displaystyle F^{-1}(\exp\{-\varphi^{-1}(U_{n+1})/Z_n\}), & Z_n>0,\\ \displaystyle 0,
& Z_n=0,
\end{cases}
\qquad n\ge 0,
\end{equation}
where $U_n$, $n\ge 1$, are i.i.d.\ random variables on $(0,1)$ and where $Z_0\ge 0$
does not depend on them.

We prove a series of MBPPLRE properties by analogy with \cite{Leb-2005d} as the
following propositions.

\begin{proposition}
If\/ $\{Z_n\}$ is an\/ $\MBPPLRE(T)$ with $F(x)$\textup, then\/ $\{\lambda Z_n\}$ for
any\/ $\lambda>0$ is an\/ $\MBPPLRE(\lambda T)$ with $F^{1/\lambda}(x/\lambda)$.
\end{proposition}

\begin{proof}
Make a substitution in \eqref{obo-mvpss-2}. Indeed,
$$
\begin{aligned}[b]
{\bf P}(\lambda Z_{n+1}\le y\,|\,\lambda Z_n=x)&={\bf P}(Z_{n+1}\le
y/\lambda\,|\,\lambda Z_n=x/\lambda)\\ &=\varphi(-(x/\lambda)\ln F(y/\lambda))\\
&=\varphi(-x\ln F^{1/\lambda}(y/\lambda)).
\end{aligned}\eqno\qedhere
$$
\end{proof}

This property implies closeness of the class $\MBPPLRE(\mathbb{R}_+)$ with respect to
multiplication by $\lambda>0$ and closeness of $\MBPPLRE(\mathbb{Z}_+)$ for
$\lambda\in \mathbb{N}$.

\begin{lemma}
For any numbers $Z'_0\le Z''_0$ and\/ $U''_n\le U''_n$\textup, $n\ge 1$\textup, with
$Z'_0, Z''_0\ge 0$ and\/ $U'_n,U''_n\in (0,1)$\textup, number sequences\/ $\{Z'_n\}$
and\/ $\{Z''_n\}$ constructed according to\/ \eqref{obo2-mvpss} satisfy the condition
$Z'_n\le Z''_n$ for all\/ $n\ge 0$.
\end{lemma}

\begin{proof}
By the condition, we have $Z'_0\le Z''_0$. Assume that $Z'_n\le Z''_n$ for some $n\ge
0$. Then, since $F^{-1}$ is a nondecreasing function, from $U'_{n+1}\le U''_{n+1}$
and equation~\eqref{obo2-mvpss} we obtain $Z'_{n+1}\le Z''_{n+1}$. Now the claim of
the lemma holds by the induction principle.
\end{proof}

Recall the notion of association of random variables \cite{EPW, Bul}.

A multivariable function $f(x)$ with $x=(x_1,\dots x_n)$ is said to be monotonically
nondecreasing if $x'_i\le x''_i$, $1\le i\le n$, implies $f(x')\le f(x'')$. Random
variables of a tuple $\zeta=(\zeta_1,\dots,\zeta_n)$ are said to be associated if
$\cov(f(\zeta),g(\zeta))\ge 0$ holds for all monotonically nondecreasing $f$ and $g$
such that this covariance exists. A random process or a field $\{\zeta(t):\: t\in
\mathcal{T}\}$ are said to be associated if their values $\zeta(t_1),
\dots,\zeta(t_n)$ are associated for any finite set $\{t_1,\dots,t_n\}\subset
\mathcal{T}$.

According to [\citen{EPW}; \citen{Bul}, Theorem~1.8], independent random variables
are always associated; monotonically nondecreasing functions of associated random
variables also possess this property.

\begin{proposition}
Any MBPPLRE is associated.
\end{proposition}

\begin{proof}
It suffices to note that, by Lemma~1, any $Z_{i_1},\dots, Z_{i_m}$ are nondecreasing
functions of the independent random variables $Z_0$ and $U_n$, $1\le
n\le\max\{i_1,\dots,i_m\}$.
\end{proof}

To establish the monotonicity in parameters, introduce a (partial) ordering relation
between distributions: $F_1\prec F_2$ if $F_1(x)\ge F_2(x)$, $\forall x$. Note that
$F_1\prec F_2$ implies $F^{-1}_1(y)\le F^{-1}_2(y)$, $\forall y\in (0,1)$.

Denote by $Z=\mathcal{Z}(F,G,H)$ an MBPPLRE with $Z_0$ having distribution $H$.

\begin{proposition}
If\/ $F'\prec F''$\textup, $G'\prec G''$\textup, and\/ $H'\prec H''$\textup, then one
can construct processes $Z'=\mathcal{Z}(F',G',H')$ and\/
$Z''=\mathcal{Z}(F'',G'',H'')$ on the same probability space so that\/ $Z'_n\le
Z''_n$ for all\/ $n\ge 0$ a.s.
\end{proposition}

\begin{proof}
Let $U_0$ be the random variable uniformly distributed on $(0,1)$ and independent
of~$U_n$, $n\ge 1$. Letting $Z'_0=(H')^{-1}(U_0)$ and $Z''_0=(H'')^{-1}(U_0)$, we
obtain $Z'_0\le Z''_0$. Assume that $Z'_n\le Z''_n$ for some $n\ge 0$. $G'\prec G''$
implies $\varphi'(u)\ge \varphi''(u)$, $\forall u>0$, and $(\varphi'(v))^{-1}\ge
(\varphi''(v))^{-1}$, $\forall v\in (0,1)$. $F'\prec F''$ implies $(F')^{-1}(y)\le
(F'')^{-1}(y)$, $\forall y\in (0,1)$. By equation~\eqref{obo2-mvpss} we obtain
$Z'_{n+1}\le Z''_{n+1}$. Proposition~3 is proved by the induction principle.
\end{proof}

Denote the limit distribution of $Z_n$ as $n\to\infty$ (if exists) by $\Psi$.

\begin{proposition}
If for two $\MBPPLRE(T)$ we have $F'\prec F''$ and\/ $G'\prec G''$\textup, then\/
$\Psi'\prec\Psi''$.
\end{proposition}

\begin{proof}
Take arbitrary $H'=H''$ on $T$ and construct processes on the same probability space
according to Proposition~3. Then $Z'_n\le Z''_n$ a.s.\ implies ${\bf P}(Z'_n\le x)\ge
{\bf P}(Z''_n\le x)$, $n\ge 1$, whence as $n\to\infty$ we obtain
$\Psi'(x)\ge\Psi''(x)$, $x>0$.
\end{proof}

\section{Ergodic Theorem}

Recall the Gumbel distribution function $\Lambda(x)=\exp\{-e^{-x}\}$, which plays an
important role in stochastic extreme value theory.

If a distribution function $F(x)$ is continuous and strictly increasing and $F(0)=0$,
then equation \eqref{obo2-mvpss} reduces by the transformation
$\zeta_n=\Lambda^{-1}(F(Z_n))$ to the form of the general (nonlinear) first-order
autoregression
\begin{equation}\label{nepr}
\zeta_{n+1}=f(\zeta_n)+\eta_{n+1},\quad n\ge 0,
\end{equation}
where $f(u)=\ln F^{-1}(\Lambda(u))$ and where independent random variables
$\eta_n=-\ln \varphi^{-1}(U_n)$, $n\ge 1$, have distribution function
$\varphi(e^{-x})$.

Ergodicity of such models has been studies, e.g., in [\citen{Borov},~Section~8.4;
\citen{Bhat}].

Denote $\delta={\bf E}\eta_1$ and assume that this mean value does exist. We have
$$
\varphi(e^{-x})={\bf E}\exp\{-\nu e^{-x}\}={\bf E}\Lambda(x-\ln\nu),
$$
which implies
\begin{equation}\label{oprdelta}
\delta=\gamma+{\bf E}\ln\nu,
\end{equation}
since the Gumbel distribution $\Lambda$ has mathematical expectation $\gamma$.

Note that $\delta$, as the mathematical expectation under the distribution function
$\varphi(e^{-x})$, can also be represented as
\begin{equation}\label{formdelta}
\delta=\int_0^{+\infty}(1-\varphi(e^{-x}))\,dx-\int_{-\infty}^0\varphi(e^{-x})\,dx=
\int_0^\infty(1-\varphi(e^{-x}))\,dx-\int_0^\infty\varphi(e^x)\,dx
\end{equation}
when all the integrals converge.

\begin{example}
Let $F(x)=\exp\{-(x/c)^{-\beta}\}$, $x,c,\beta>0$ (Fr\'echet distribution). Then the
MBPPLRE admits a constructive representation
\begin{equation}\label{zw}
Z_{n+1}=W_{n+1}Z_n^{1/\beta},
\end{equation}
where $W_n$, $n\ge 1$, are independent and have distribution function
$\varphi((x/c)^{-\beta})$, and \eqref{nepr} can be rewritten in the linear
autoregression form
\begin{equation}\label{zb}
\zeta_{n+1}=\zeta_n/\beta+\ln c+\eta_{n+1},\quad n\ge 0.
\end{equation}

For $\beta<1$, the process $\{\zeta_n\}$ goes to $\pm\infty$ depending on the sign of
the initial condition~$\zeta_0$. For $\beta>1$, the process $\{\zeta_n\}$ is ergodic.
For $\beta=1$, we have a simple random walk going to $+\infty$ when $c>e^{-\delta}$,
to $-\infty$ when $c<e^{-\delta}$, and oscillating between $\pm\infty$ when
$c=e^{-\delta}$.

From this, one can easily obtain results for $\{Z_n\}$ taking into account that
$Z_n=F^{-1}(\Lambda(\zeta_n))$, whence $Z_n\to 0$ for $\zeta_n\to -\infty$ and
$Z_n\to +\infty$ for $\zeta_n\to +\infty$, and ergodicity is preserved.
\end{example}

This example suggests that the following theorem holds.

\begin{theorem}
If\/ $\MBPPLRE(T)$\textup, $T\subset (0,+\infty)$\textup, satisfies the conditions
\begin{equation}\label{usl1}
\liminf_{x\to+\infty} x(-\ln F(x))<e^{-\delta}
\end{equation}
and
\begin{equation}\label{usl2}
\liminf_{x\to 0} x(-\ln F(x))>e^{-\delta},
\end{equation}
where\/ $\delta$ in\/ \eqref{oprdelta} exists\textup, then the process is ergodic.
\end{theorem}

\begin{proof}
Without loss of generality we may assume that $T=(0,+\infty)$. First of all, note
that conditions \eqref{usl1} and \eqref{usl2} are equivalent to the claim that there
exist $0<x_1<x_2$ and $0<c_2<e^{-\delta}<c_1$ such that $F(x)\le e^{-c_1/x}$ for
$x\le x_1$ and $F(x)\ge e^{-c_2/x}$ for $x\ge x_2$.

As a Lyapunov function, consider $g(x)=(\ln(x/x_2))_++(\ln(x_1/x))_+$ with
$y_+=\max\{0,y\}$. Denote{\sloppy
$$
\mu(x)={\bf E}(g(Z_{n+1})\,|\,Z_n=x)-g(x).
$$
Then}
\begin{align*}
\mu(x)&={\bf E}\{\ln(Z_{n+1}/x_2)_+\,|\,Z_n=x\}+ {\bf
E}\{\ln(x_1/Z_{n+1})_+\,|\,Z_n=x\}-g(x)\\ &=\int_0^\infty{\bf
P}(\ln(Z_{n+1}/x_2)>y\,|\,Z_n=x)\,dy+ \int_0^\infty{\bf
P}(\ln(x_1/Z_{n+1})>y\,|\,Z_n=x)\,dy-g(x)\\ &=\int_0^\infty(1-\varphi(-x\ln
F(x_2e^y)))\,dy+\int_0^\infty\varphi(-x\ln F(x_1e^{-y})))\,dy-g(x)\\
&\le\int_0^\infty(1-\varphi((c_2x/x_2)e^{-y}))\,dy
+\int_0^\infty\varphi((c_1x/x_1)e^y))\,dy-g(x)\\
&=\int_{-\ln(c_2x/x_2)}^\infty(1-\varphi(e^{-z}))\,dz
+\int_{\ln(c_1x/x_1)}^\infty\varphi(e^z)\,dz-g(x)\\
&=\int_0^\infty(1-\varphi(e^{-z}))\,dz+\int_0^\infty\varphi(e^z)\,dz+
\int_{-\ln(c_2x/x_2)}^0(1-\varphi(e^{-z}))\,dz\\ &\quad\strut
-\int_0^{\ln(c_1x/x_1)}\varphi(e^z)\,dz -((\ln(x/x_2))_++(\ln(x_1/x))_+).
\end{align*}
Denote the obtained upper estimate for $\mu(x)$ by $\mu^*(x)$, and let
$$
{\tilde\mu}(x)=\int_{-\ln(c_2x/x_2)}^0(1-\varphi(e^{-z}))\,dz
-\int_0^{\ln(c_1x/x_1)}\varphi(e^z)\,dz -((\ln(x/x_2))_++(\ln(x_1/x))_+).
$$
Then
\begin{equation}\label{formmu}
\mu(x)\le \mu^*(x)=\int_0^\infty(1-\varphi(e^{-z}))\,dz
+\int_0^\infty\varphi(e^{-z})\,dz+{\tilde\mu}(x).
\end{equation}

For $x\ge x_2$, we have
\begin{align*}
{\tilde\mu}(x)&=\int_{-\ln(c_2x/x_2)}^0(1-\varphi(e^{-z}))\,dz
-\int_0^{\ln(c_1x/x_1)}\varphi(e^z)\,dz-\ln(c_2x/x_2)+\ln c_2\\
&=-\int_{-\ln(c_2x/x_2)}^0\varphi(e^{-z})\,dz-\int_0^{\ln(c_1x/x_1)}\varphi(e^z)\,dz
+\ln c_2\\ &\to -2\int_0^\infty\varphi(e^z)\,dz+\ln c_2,\quad x\to+\infty,
\end{align*}
whence by \eqref{formdelta} and \eqref{formmu} we obtain
$\mu(x)\le\mu^*(x)\to\delta+\ln c_2=\ln (c_2/e^{-\delta})<0$, $x\to +\infty$.

For $x\le x_1$, we have
\begin{align*}
{\tilde\mu}(x)&=\int_{-\ln(c_2x/x_2)}^0(1-\varphi(e^{-z}))\,dz
-\int_0^{\ln(c_1x/x_1)}\varphi(e^z)\,dz+\ln(c_1x/x_1)-\ln c_1\\
&=-\int_{-\ln(c_2x/x_2)}^0\varphi(e^{-z})\,dz
-\int_0^{-\ln(c_1x/x_1)}(1-\varphi(e^{-z}))\,dz -\ln c_1\\ &\to
-2\int_0^\infty(1-\varphi(e^{-z}))\,dz-\ln c_2,\quad x\to 0,
\end{align*}
whence by \eqref{formdelta} and \eqref{formmu} we obtain $\mu(x)\le\mu^*(x)\to
-\delta-\ln c_1=-\ln (c_1/e^{-\delta})<0$, $x\to 0$.

Hence, there exist $\varepsilon>0$ and $0<v_1\le v_2$ such that
$\mu(x)\le-\varepsilon$ for $x\notin V=[v_1,v_2]$. Furthermore, $\sup_{x\in V}{\bf
E}(g(Z_{n+1})\,|\,Z_n=x)<\infty$. Thus, the Lyapunov conditions
\cite[Section~4.2]{Borov} are satisfied.

Now let us check the mixing condition.

Let $0<u_1\le u\le u_2$, $0<a<b$. Since $\varphi(s)$ is convex and decreasing, we
have
$$
\varphi(-u\ln F(b))-\varphi(u\ln F(a))\ge\frac{u_1}{u_2}(\varphi(-u_2\ln
F(b))-\varphi(-u_2\ln F(a)),
$$
which implies that for any measurable $B$ we have
\begin{equation}\label{perem}
{\bf P}(Z_{n+1}\in B\,|\,Z_n=x)\ge\frac{v_1}{v_2}{\bf P}(Z_{n+1}\in B\,|\,Z_n=v_2),
\quad \forall x\in V.
\end{equation}

Furthermore, any MBPPLRE is irreducible and aperiodic (in other words, from any state
$x\in T$ one can get into any set $B\subset T$ and, moreover, in one step). Now the
Lyapunov conditions  and \eqref{perem} imply the ergodicity of the MBPPLRE according
to \cite[Section~2, Theorem~2]{Borov}.
\end{proof}

Note that if $\nu=1$ a.s., Theorem~1 reduces to Theorem~A, since in this case
$\delta=\gamma$.

In some cases it is more convenient to compute $\delta$ by the definition than by
equation~\eqref{oprdelta}.

\begin{example}
Let $\nu$ have exponential distribution with mean $\theta$. Then
$\varphi(u)=(1+\theta u)^{-1}$ and
$$
\varphi(e^{-x})=\frac{1}{1+\theta e^{-x}}=\frac{1}{1+e^{-(x-\ln\theta)}};
$$
i.e., we obtain a logistic distribution with shift parameter $\ln\theta$. Hence,
$\delta=\ln\theta$.
\end{example}

\begin{example}
Let $\nu$ have a strictly stable distribution with $\varphi(u)=e^{-cu^\alpha}$,
$c>0$, $0<\alpha<1$. Then
$$
\varphi(e^{-x})=\exp\{-ce^{-\alpha x}\}=\Lambda(\alpha x-\ln c),
$$
whence
$$
\delta=\frac{\gamma+\ln c}{\alpha}.
$$
\end{example}

This, by the way, provides a convenient method for computing the mean logarithm of a
stable random variable. Using \eqref{oprdelta}, we obtain
$$
{\bf E}\ln\nu=\frac{\gamma(1-\alpha)+\ln c}{\alpha}.
$$

For an ergodic $\MBPPLRE(T)$, denote the random variable with the limit distribution
by ${\tilde Z}$. In some cases, we can find numerical characteristics of this
distribution.

\begin{example}
Let $\nu$ have a strictly stable distribution with $\varphi(u)=e^{-u^\alpha}$,
$0<\alpha<1$, and $F(x)=\exp\{-x^{-\beta}\}$, $x,\beta>0$. Then we obtain
representation \eqref{zw}, where the $W_n$, $n\ge 1$, have the Fr\'echet distribution
function  $\exp\{-x^{-\alpha\beta}\}$, $x>0$, and
$$
{\bf E}W_1^s=\Gamma\left(1-\frac{s}{\alpha\beta}\right),\quad 0<s<\alpha\beta.
$$
Representation \eqref{zw} and ergodicity of the process imply that the limit
distribution is the distribution of the following infinite product, which converges
a.s.:
$$
{\tilde Z}\stackrel{d}{=}\prod_{n=0}^\infty W_n^{1/\beta^n},
$$
whence
$$
{\bf E}{\tilde Z}^s=\prod_{n=1}^\infty
\Gamma\left(1-\frac{s}{\alpha\beta^n}\right),\quad 0<s<\alpha\beta.
$$
\end{example}

Clearly, Proposition 4 also holds in the cases where one or both of the limit
distributions is/are concentrated at zero. From this, in particular, one can obtain
the degeneracy condition for processes with $F(0)>0$. Here, by the degeneracy we mean
vanishing of the process starting from some (random) moment.

\begin{corollary}
If an\/ $\MBPPLRE(\mathbb{R}_+)$ satisfies\/ \eqref{usl1} and\/ $F(0)>0$\textup, then
the process degenerates a.s.
\end{corollary}

\begin{proof}
Note that for any $C>0$ and $n\ge 0$ by equation \eqref{obo-mvpss-2} we have
$$
{\bf P}(Z_{n+1}=0)={\bf E}\varphi(-Z_n\ln F(0))\ge {\bf P}(Z_n=0)+\varphi(-C\ln
F(0)){\bf P}(0<Z_n\le C).
$$
The sequence ${\bf P}(Z_n=0)$ is monotone nondecreasing and bounded, and therefore
tends to some limit $p_0\in (0,1]$. We obtain
$$
\sum_{n=0}^\infty{\bf P}(0<Z_n\le C)\le p_0/\varphi(-C\ln F(0))<\infty,
$$
so by the Borel--Cantelli lemma $Z_n$ gets into $(0,C]$ finitely many times a.s.\ for
any $C>0$, which may mean either degeneracy or going to infinity as $n\to\infty$.

Let $F^*(x)=F(x)\exp\{-x^{-2}\}{\bf I}(x>0)$; then $F\prec F^*$. By Proposition~3,
one can construct an $\MBPPLRE(\mathbb{R}_+)$ with $F^*$ such that $Z_n\le Z^*_n$
a.s\@. Note that $F^*$ satisfies the conditions of Theorem~1, so that ${\bf
P}(Z^*_n\to+\infty)=0$, and hence ${\bf P}(Z_n\to+\infty)=0$ too. Thus, the process
degenerates a.s.
\end{proof}

Up to now, we assumed that $\delta$ is finite. However, a limit distribution may in
some cases exist as well for $\delta=+\infty$, which corresponds to super-heavy tails
$G$.

\begin{example}
Let $F(x)=\exp\{-x^{-\beta}\}$, $x>0$, $\beta>1$, and $G(x)=1-1/\ln x$, $x\ge e$.
Then by the Tauber theorem we have $1-\varphi(u)\sim -1/\ln u$, $u\to 0$, whence
\begin{equation}\label{etaphi}
{\bf P}(\eta_1>x)=1-\varphi(e^{-x})\sim 1/x,\quad x\to+\infty.
\end{equation}
On the other hand, since $\nu\ge e$ a.s., we have $\varphi(u)\le e^{-eu}$, $u\ge 0$,
whence
\begin{equation}\label{etaphi2}
{\bf P}(\eta_1<-x)=\varphi(e^x)\le \exp\{-e^{x+1}\},\quad x>0.
\end{equation}

It follows from \eqref{zb} that existence of a limit distribution is equivalent to
convergence of the following random series a.s.:
\begin{equation}\label{rzb}
{\tilde\zeta}\stackrel{d}{=}\sum_{n=0}^\infty\beta^{-n}\eta_n,
\end{equation}
where ${\tilde\zeta}=\Lambda^{-1}(F({\tilde Z}))$. For any $1/\beta<\varepsilon<1$ by
virtue of \eqref{etaphi} and \eqref{etaphi2} we obtain
$$
\sum_{n=0}^\infty {\bf P}(|\eta_n|>(\varepsilon\beta)^n)<\infty,
$$
whence it follows by the Borel--Cantelli lemma that the events
$A_n=\{\beta^{-n}|\eta_n|>\varepsilon^n\}$ occur at most finitely many times; hence,
the series \eqref{rzb} converges a.s.
\end{example}

\section{Conclusion}

We have introduced maximal branching processes in random environment (with a single
particle type). We have examined the case of ``power-law'' random environment; for
this case we have studied a number of properties, proved an ergodic theorem, and
considered examples. We have noted a relation between maximal branching process and
infinite-server queues. Further research can address the analysis of a wider class of
random environments and processes with immigration.

\end{document}